\documentclass[12pt]{amsart}
\usepackage{colortbl}
\usepackage{longtable}
\usepackage{hyperref}
\usepackage{amsmath}
\usepackage{pslatex}
\usepackage{amsthm}
\usepackage{amssymb}
\usepackage{latexsym}
\usepackage{graphicx}
\usepackage{multicol}
\usepackage{tikz}
\usepackage{pst-plot}
\usepackage{pstricks}

\hypersetup{hypertex, colorlinks=true, linkcolor=blue, filecolor=blue, pagecolor=blue, urlcolor=blue}

\newtheorem{theorem}{Theorem}[section]
\newtheorem{lemma}[theorem]{Lemma}
\newtheorem{corollary}[theorem]{Corollary}
\newtheorem{conjecture}[theorem]{Conjecture}
\theoremstyle{definition}

\numberwithin{equation}{section}

\author{Skyler Simmons}
\address{
Mathematics Department\\
Utah Valley University\\
Orem, UT, 84058
}
\email{skyler.simmons@uvu.edu}

\begin{document}

\title[Stability of Broucke's Orbit]{Stability of Broucke's Isosceles Orbit}

\maketitle

\begin{abstract}
We extend the result of Yan to Broucke's isosceles orbit with masses $m_1$, $m_1$, and $m_2$ with $2m_1 + m_2 = 3$.  Under suitable changes of variables, isolated binary collisions between the two mass $m_1$ particles are regularizable.  We analytically extend a method of Roberts to perform linear stability analysis in this setting.  Linear stability is reduced to computing three entries of a $4 \times 4$ matrix related to the monodromy matrix.  Additionally, it is shown that the four-degrees-of-freedom setting has a two-degrees-of-freedom invariant set, and linear stability results in the subset comes ``for free'' from the calculation in the full space.  The final numerical analysis shows that the four-degrees-of-freedom orbit is linearly unstable except for the interval $0.595 < m_1 < 0.715$, whereas the two-degrees-of-freedom orbit is linearly stable for a much wider interval.

\end{abstract}

\section{Introduction} %%THIS NEEDS TO BE RE-DONE

Mathematically, the study of determining the motion of $n$ point masses in space whose motion is governed by Newton's gravitational law is known as the Newtonian $n$-body problem.  Notationally, if $\{q_1, q_2, ..., q_n\}$ represent the positions of the bodies in $\mathbb{R}^k$ ($k = 1, 2,$ or $3$) with masses $\{m_1, m_2, ..., m_n\}$ respectively, then their motion is governed by the system of differential equations
\begin{equation}
m_i \ddot{q}_i  = \sum_{i \neq j} \frac{m_i m_j (q_j - q_i)}{|q_i - q_j|^3},
\label{standardmotion}
\end{equation}
where the dot represents the derivative with respect to time.  Despite hundreds of years of study and the relatively recent development of computer ODE solvers, many open questions about the $n$-body problem remain.  \\

One aspect of the $n$-body problem that has been getting much attention of late are orbits involving collision singularities.  A \textit{collision singularity} occurs when $q_i = q_j$ for some $i \neq j$.  In the equations governing the motion, this results in a zero denominator in one or more terms in the sum.  Under certain conditions, these collisions can be \textit{regularized} and the solutions can be continued past collision. \\

In this paper, we will study Broucke's isosceles triangle orbit, originally presented in \cite{bibBrouckeOriginal}.  This is among the earliest-introduced orbits featuring collisions in the planar setting.  The orbit features regularizable collisions between two of the bodies, while the third oscillates along the vertical axis.  \\

The chief aim of this paper is to give stability results, including linear stability, for the orbit being considered in both a four-degrees-of-freedom (4DF) and a two-degrees-of-freedom (2DF) setting at the same time.  With our choice of coordinates in the 4DF setting, the linearized phase space of the regularized equations has an elegant decomposition into two invariant subspaces.  The linear stability analysis in the 2DF setting corresponds to one of these subspaces.  Additionally, we are able to use analytical techniques to reduce the linear stability analysis of the numerical calculation of three entries of a $4 \times 4$ matrix $K$ related to the monodromy matrix.  This type of stability calculation based on an invariant-subspace decomposition was done earlier in \cite{bibBSrhomb}.  A suitable coordinate transformation reveals this same structure in the Broucke's setting.  \\

Schubart \cite{bibSchubart} was one of the first to study periodic orbits with regularizable collisions.  He was able to find a collinear three-body equal-mass orbit where the central body alternated between collisions with the outer two.  This was further extended to the case of arbitrary masses numerically by H\'enon \cite{bibHenon} in 1977.  Analytic existence of the equal-outer-mass orbit was established independently by Venturelli \cite{bibVenturelli} and Mockel \cite{bibMoeckel}, both in 2008.  Shibiyama \cite{bibShib1} recently demonstrated the existence of the arbitrary-mass version.  The study of linear stability of Schubart's orbit was performed by Hietarinta and Mikkola \cite{bibHM1} in 1993.\\

Sweatman found a Schubart-like collinear four-body symmetric orbit in 2002 \cite{bibSweatman1}, and later studied its linear stability \cite{bibSweatman2}.  This orbit features simultaneous binary collisions between two outer pairs of bodies followed by an interior collision between the two central bodies.  Analytic existence of this orbit was given by Ouyang and Yan in \cite{bibYan3}.  \\

Apart from the Broucke's orbit, other planar orbits with regularizable collisions have also been studied.  Among the first of these is the rhomboidal four-body orbit, which features two pairs of bodies: one pair on the $x$-axis and the second on the $y$-axis.  The pairs collide at the origin in an alternating fashion.  This orbit was shown to exist analytically in multiple independent papers (by Yan in \cite{bibYan1} and Martinez in \cite{bibMartinez} for equal masses, \cite{bibShib1} for symmetric masses).  Additionally, Yan showed that for equal masses, the orbit is linearly stable.  This was followed up by work in \cite{bibBSrhomb}, in which linear stability was shown for a wide range of mass ratios.\\

%In \cite{bibERNESTOS}, Lacomba and P\'erez-Chavela adapt the coordinate transformation of McGehee to analyze the total-collapse manifolds of the orbit.  \\

Other planar orbits with singularities have also been studied.  A planar four-body orbit featuring simultaneous binary collisions was described in \cite{bibOYS}.  The orbit was shown to be linearly stable in \cite{bibBORSY}.  It was later shown that this orbit could be numerically extended to symmetric masses in \cite{bibBOYS1} (see also \cite{bibBOYS2}), and linear stability for this extension was shown for an interval of certain mass ratios in \cite{bibBMS}.  \\

More generally, analytic existence of large families of orbits with two degrees of freedom and regularizable singularities was recently proven by Shibayama in \cite{bibShib1} and Martinez in \cite{bibMartinez}. \\

The remainder of the paper is as follows:  In Section \ref{describe}, we describe the orbit and some of its properties.  In \ref{setting}, we give formal notation describing the orbit.  We also perform the coordinate transformations that regularize the collisions.  Section \ref{aExist} gives an analytic proof of the existence of the orbit in the regularized 4DF setting. \\

Section \ref{lsstart} deals with the linear stability of the orbit.  In \ref{theoryreview}, we review some of the basic properties of linear stability.  Next, \ref{syms} describes the symmetries of the orbit, which are needed to perform the analysis in \ref{Roberts}.  Section \ref{Kapriori} describes all of the remaining linear stability analysis that can be done before any numerical work, including the decomposition into two invariant sets mentioned earlier. \\

Finally, in Section \ref{Numerics}, we present the results obtained from the numerical calculations of stability.\\

\section{Broucke's Isosceles Orbit}\label{describe}

\subsection{Setting and Regularization}\label{setting}

We denote the position of the three bodies with the variables $\{q_i\}$.  The two bodies at $(q_1, q_2)$ and $(q_3, q_4)$ are given mass $m_1$, and the third body has mass $m_2$, where both $m_1$ and $m_2$ are non-negative real numbers.  Also, $p_i = m_1\dot{q}_i$ for $i= 1,2,3,4$, and $p_i = m_2\dot{q}_i$ for $i = 5,6$.  (See Figure \ref{figTheOrbit}.)\\

\begin{figure}[h]
\begin{center}
\begin{tikzpicture}[scale = .5]
\draw[<->] (-5,0)--(5,0);
\draw[<->] (0,-5)--(0,5);
%\draw[->] (14,0)--(14,10) node[above] {$Q = Q_0$};
\draw[fill] (2,3) circle [radius=.25] node[right] {$(q_1,q_2)$};
\draw[fill] (-2,3) circle [radius=.25] node[left] {$(q_3,q_4)$};
\draw[fill] (0,-4) circle [radius=.25] node[right] {$(q_5,q_6)$};
\draw[dotted] (2,3)--(-2,3);
\draw[dotted] (2,3)--(0,-4);
\draw[dotted] (-2,3)--(0,-4);
\end{tikzpicture}
\caption{The coordinates for Broucke's Orbit.}
\label{figTheOrbit}
\end{center}
\end{figure}
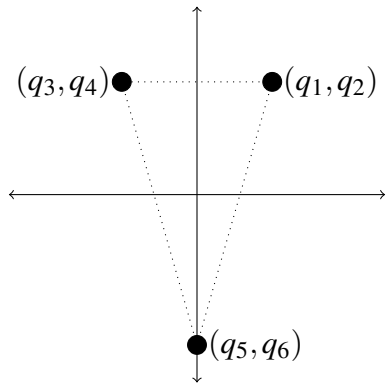

The orbit maintains the shape of an isosceles triangle for all time while alternating between collisions of the two bodies of mass $m_1$, as shown below.  Beginning at collision on the positive $y$-axis (shown by the lightest-colored circles in Figure \ref{figTheOrbit2}), the two bodies of mass $m_1$ move outward in a symmetric fashion as the mass $m_2$ body moves upward along the negative $y$-axis.  This eventually leads to a collinear configuration along the $x$-axis (darkest circles).  The motion is then continued in reverse with the vertical coordinates negated, leading to a collision of the two mass $m_1$ bodies on the negative $y$-axis.  After this collision, the bodies move back along the same paths leading back to a collision of the two mass $m_1$ bodies on the positive $y$-axis.  Throughout, the mass $m_2$ body oscillates along the $y$-axis.  Tracing out the full path the bodies take yields a shape resembling the letter $\Phi$.  \\

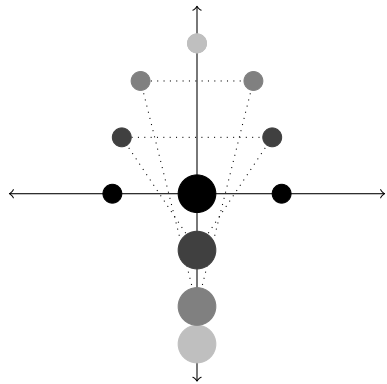
\begin{figure}[h]
\begin{center}
\begin{tikzpicture}[scale = .5]

\draw[<->] (-5,0)--(5,0);
\draw[<->] (0,-5)--(0,5);

\draw[dotted] (1.5,3)--(-1.5,3);
\draw[dotted] (1.5,3)--(0,-3);
\draw[dotted] (-1.5,3)--(0,-3);

\draw[dotted] (2,1.5)--(-2,1.5);
\draw[dotted] (2,1.5)--(0,-1.5);
\draw[dotted] (-2,1.5)--(0,-1.5);

\draw[fill, lightgray] (0,4) circle [radius=.25];
\draw[fill, lightgray] (0,4) circle [radius=.25];
\draw[fill, lightgray] (0,-4) circle [radius=.5];

\draw[fill, gray] (1.5,3) circle [radius=.25];
\draw[fill, gray] (-1.5,3) circle [radius=.25];
\draw[fill, gray] (0,-3) circle [radius=.5];

\draw[fill, darkgray] (2,1.5) circle [radius=.25];
\draw[fill, darkgray] (-2,1.5) circle [radius=.25];
\draw[fill, darkgray] (0,-1.5) circle [radius=.5];

\draw[fill, black] (2.25,0) circle [radius=.25];
\draw[fill, black] (-2.25,0) circle [radius=.25];
\draw[fill, black] (0,0) circle [radius=.5];

\end{tikzpicture}
\caption{Broucke's orbit as it evolves in time.  One-quarter period is shown.  The remainder of the orbit is obtained by a symmetric extension.}
\label{figTheOrbit2}
\end{center}
\end{figure}

For simplicity, we will assume that $2m_1 + m_2 = 3$ throughout.  Thus, the equal-mass case occurs when $m_1 = m_2 = 1$. \\

The angular momentum is given by
$$A_0 = q_1p_2 - q_2p_1 + q_3p_4 - q_4p_3 + q_5p_6 - q_6p_5.$$
Assuming the center of mass is located at the origin and net momentum is zero yields
$$q_5 = -\frac{m_1(q_1 + q_3)}{m_2}, \quad q_6 = -\frac{m_1(q_2 + q_4)}{m_2},$$
$$p_5 = -(p_1 + p_3), \quad p_6 = -(p_2 + p_4).$$

The Hamiltonian function in this setting, which we will denote $H_0$, is given by $H_0 = K_0 - U_0$, where
$$K_0 = \frac{1}{2}\left(\frac{p_1^2}{m_1} + \frac{p_2^2}{m_1} + \frac{p_3^2}{m_1} + \frac{p_4^2}{m_1} + \frac{(p_1+p_3)^2}{m_2} + \frac{(p_2 + p_4)^2}{m_2}\right)$$
and
\begin{align*}
U_0 &= \frac{m_1^2}{\left((q_1 - q_3)^2 + (q_2 - q_4)^2 \right)^{\frac{1}{2}}} \\
&+ \frac{m_1m_2}{\left(\left(q_1 + \frac{m_1(q_1 + q_3)}{m_2}\right)^2 + \left(q_2 + \frac{m_1(q_2 + q_4)}{m_2}\right)^2 \right)^{\frac{1}{2}}} \\
&+ \frac{m_1m_2}{\left(\left(q_3 + \frac{m_1(q_1 + q_3)}{m_2}\right)^2 + \left(q_4 + \frac{m_1(q_2 + q_4)}{m_2}\right)^2 \right)^{\frac{1}{2}}}.
\end{align*}

We wish to regularize the collision between the two particles of mass $m_1$.  We define a first transformation of variables $(q_i, p_i) \leftrightarrow (u_i, v_i)$ by means of the generating function
$$F_1 = \left(\frac{u_1 + u_3}{2}\right)p_1 + \left(\frac{u_2 + u_4}{2}\right)p_2 + \left(\frac{u_3 - u_1}{2}\right)p_3 + \left(\frac{u_4 - u_2}{2}\right)p_4.$$
Then, with $q_i = \partial F_1/\partial p_i$ and $v_i = \partial F_1/\partial u_i$,  we have
\begin{align*}
q_1 &= \frac{u_1 + u_3}{2} \ &v_1 &= \frac{p_1 - p_3}{2} \\
q_2 &= \frac{u_2 + u_4}{2} \ &v_2 &= \frac{p_2 - p_4}{2} \\
q_3 &= \frac{u_3 - u_1}{2} \ &v_3 &= \frac{p_1 + p_3}{2} \\
q_4 &= \frac{u_4 - u_2}{2} \ &v_4 &= \frac{p_2 + p_4}{2}. \\
\end{align*}
Solving for each $p_i$ gives
$$p_1 = v_1 + v_3, \quad p_2 = v_2 + v_4, \quad p_3 = v_3 - v_1, \quad p_4 = v_4 - v_2.$$
Substituting, the new Hamiltonian $H_1$ is given by $K_1 - U_1$, where
$$K_1 = \frac{1}{m_1}(v_1^2 + v_2^2 + v_3^2 + v_4^2) + \frac{2}{m_2}(v_3^2 + v_4^2)$$
and setting $\mu = \frac{1}{2} + \frac{m_1}{m_2}$, we have
\begin{align*}
U_1 &= \frac{m_1^2}{\left(u_1^2 + u_2^2 \right)^{\frac{1}{2}}} \\
&+ \frac{m_1m_2}{\left( \left(\frac{1}{2}u_1 + \mu u_3\right)^2 + \left(\frac{1}{2}u_2 +  \mu u_4\right)^2 \right)^{\frac{1}{2}}} \\
&+ \frac{m_1m_2}{\left( \left(\mu u_3 - \frac{1}{2}u_1\right)^2 + \left(\mu u_4 - \frac{1}{2} u_2\right)^2 \right)^{\frac{1}{2}}}.
\end{align*}
The angular momentum after this transformation is
$$A_1 = u_1v_2 - u_2v_1 + 2\mu(u_3v_4 - u_4v_3).$$

\textit{Remark:} This transformation ``codes in'' some of the symmetry present in the orbit.  In particular, when $u_2 = 0$ we have that $q_2 - q_4 = 0$, which places the two bodies of mass $m_1$ on the same horizontal line.  Similarly, if $u_3 = 0$ then $q_1 + q_3 = 0$, which places those same bodies symmetrically across the $y$-axis.  Similar results are present for $v_2  0$ and $v_3 = 0$.  As these behaviors are expected along the orbit, these are natural targets for the coordinate transform.  This also gives a similar setting to what is present in the rhomboidal problem discussed in \cite{bibBSrhomb}, wherein ``symmetry-preserving zeros'' facilitated analysis of the orbit. \\

The second transformation $(u_i, v_i) \leftrightarrow (Q_i, P_i)$ is generated by
$$F_2 = v_1(Q_1^2 - Q_2^2) + 2v_2Q_1Q_2 +  v_3Q_3 + v_4Q_4.$$
With $u_i = \partial F_2 / \partial v_i$ and $P_i = \partial F_2 / \partial Q_i$, we have

\begin{align*}
u_1 &= Q_1^2 - Q_2^2 \ &P_1 &= 2v_1Q_1 - 2v_2Q_2 \\
u_2 &= 2Q_1Q_2 \ & P_2 &= 2v_1Q_2 + 2v_2Q_1\\
u_3 &= Q_3 \ &P_3 &= v_3\\
u_4 &= Q_4 \ &P_4 &= v_4\\
\end{align*}

Solving for $v_1$ and $v_2$ gives
$$v_1 = \frac{Q_1P_1 - Q_2P_2}{2(Q_1^2 + Q_2^2)}, \quad v_2 = \frac{Q_2P_1 + Q_1P_2}{2(Q_1^2 + Q_2^2)}.$$

Making the substitution gives the new Hamiltonian $H_2 = K_2 - U_2$, with
$$K_2 = \frac{P_1^2 + P_2^2}{4m_1(Q_1^2 + Q_2^2)} + \left(P_3^2 + P_4^2\right)\left(\frac{1}{m_1} + \frac{2}{m_2}\right)$$
and (after some simplification)
\begin{align*}
U_2 &= \frac{m_1^2}{Q_1^2 + Q_2^2} \\
&+ \frac{m_1m_2}{\sqrt{\frac{1}{4}(Q_1^2+Q_2^2)^2 + \mu\mathbf{Q} + \mu^2(Q_3^2 + Q_4^2)}} \\
&+ \frac{m_1m_2}{\sqrt{\frac{1}{4}(Q_1^2+Q_2^2)^2 - \mu\mathbf{Q} + \mu^2(Q_3^2 + Q_4^2)}},
\end{align*}
where $\mathbf{Q} = (Q_1^2 - Q_2^2)Q_3 + 2Q_1Q_2Q_4$.

The angular momentum is now
\begin{equation}
\label{angMomFinal}
A_2 = \frac{1}{2}\left(Q_1P_2 - Q_2P_1\right) + 2\mu(Q_3P_4 - Q_4P_3).
\end{equation}

As a final step, we create the regularized Hamiltonian $\Gamma = \frac{dt}{ds}(H_2 - E)$, where $\frac{dt}{ds} = Q_1^2 + Q_2^2$.  Then
\begin{align*}
\Gamma &= \frac{P_1^2 + P_2^2}{4m_1} + \left(P_3^2 + P_4^2\right)\left(Q_1^2 + Q_2^2\right)\left(\frac{1}{m_1} + \frac{2}{m_2}\right) \\
&- \frac{m_1m_2\left(Q_1^2 + Q_2^2\right)}{\sqrt{\frac{1}{4}(Q_1^2+Q_2^2)^2 + \mu\mathbf{Q} + \mu^2(Q_3^2 + Q_4^2)}} \\
&- \frac{m_1m_2\left(Q_1^2 + Q_2^2\right)}{\sqrt{\frac{1}{4}(Q_1^2+Q_2^2)^2 - \mu\mathbf{Q} + \mu^2(Q_3^2 + Q_4^2)}} \\
&- m_1^2 - E(Q_1^2 + Q_2^2).
\end{align*}
It is easy to verify that when $Q_1^2 + Q_2^2 = 0$, the condition $\Gamma = 0$ forces $P_1^2 + P_2^2 = 4m_1^3$.  Hence, the binary collisions between the two equal-mass particles has been regularized.  Attempting to regularize the other binary collisions is not needed for this orbit, and is in general a difficult problem (see \cite{bibMoeckelMontgomery}).  Also, regularization of the triple collision where all $Q_i = 0$ (corresponding to total collapse of the system at the origin) is generally not possible (see \cite{bibMcGehee}).  

\subsection{Analytic Existence of the Orbit}\label{aExist}

In \cite{bibBrockeYan}, existence of Broucke's orbit was established in a reduced, two-degree-of-freedom setting with $m_1 = 1$.  We demonstrate the existence of the orbit in the larger, four-degree-of-freedom setting established here.

\begin{theorem}
Broucke's isosceles orbit described in Section \ref{setting} analytically exists for the Hamiltonian system given by $\Gamma$.
\end{theorem}

\begin{proof}
Let $\mathcal{A}$ denote the set where
$$q_1 = -q_3\text{, } q_2 = q_4\text{, } p_1 = -p_3\text{, and } p_2 = p_4.$$
This corresponds to the setting used in \cite{bibBrockeYan}, wherein a symmetry across the $y$-axis is forced for all time.  In that paper, it was assumed that $m_2$ (corresponding to the body running along the $y$-axis) was arbitrary, and $m_1$ (corresponding to the mass of the colliding pair) was fixed to be 1.  With a suitable re-scaling of mass ratios, the proof in \cite{bibBrockeYan} shows that the orbit exists in this setting. \\

In the set $\mathcal{A}$, we have $u_2 = u_3 = v_2 = v_3 = 0$ in the first coordinate transformation.  In turn, $u_3 = v_3 = 0$ forces $Q_3 = P_3 = 0$.  The condition $u_2 = 0$ requires that either $Q_1 = 0$ or $Q_2 = 0$.  Additionally, since $v_2 = 0$, we have that $P_2 = 2v_1Q_2$ and $P_1 = 2v_1Q_1$.  In keeping with our analysis from \cite{bibBSrhomb}, we will make the choice to set $Q_2 = 0$.  Then, when $\mathcal{A}$ holds, we have
$$Q_2 = Q_3 = P_2 = P_3 = 0$$
\noindent Furthermore,  we have
$$\dot{Q}_2 \big|_{\mathcal{A}} = \dot{Q}_3 \big|_{\mathcal{A}} = \dot{P}_2 \big|_{\mathcal{A}} = \dot{P}_3 \big|_{\mathcal{A}} =0,$$
\noindent so $\mathcal{A}$ is invariant.  Since analytic existence of the 2DF was given in \cite{bibBrockeYan}, analytic existence in the 4DF setting follows from the invariance of $\mathcal{A}$ following a rescaling of the masses.
\end{proof}

As a further consequence of the invariant set $\mathcal{A}$, initial conditions for the periodic orbit in the 2DF setting automatically give the initial conditions for the 4DF setting.  This is very useful numerically as it reduces the number of calculations required to find initial conditions.

\section{Linear Stability Analysis}\label{lsstart}

The material in this section is very similar to that of \cite{bibBSrhomb}, which in turn draws on material in \cite{bibRoberts1}. We repeat it here for completeness.

\subsection{Linear Stability Theory}\label{theoryreview}
We will begin this section by briefly reviewing the basic definitions for linear stability.  Let $\Gamma$ be a smooth function defined on an open subset of $\mathbb{R}^8$.  Let $\gamma(s)$ be a $T$-periodic solution of the system $z' = JD\Gamma(z)$, where $' = d/ds$,
\begin{equation*}
J = 
\begin{bmatrix}
O & I \\
-I & O \\
\end{bmatrix},
\end{equation*}
and $I$ and $O$ are the $4 \times 4$ identity and zero matrices, respectively.  If $X(s)$ is the fundamental matrix solution of the linearized equations

\begin{equation}
\xi' = JD^2\Gamma(\gamma(s))\xi, \quad \xi(0) = I
\label{linearized1}
\end{equation}
then the monodromy matrix is given by $X(T)$ and satisfies $X(s + T) = X(s)X(T)$ for all $s$.  Eigenvalues of the monodromy matrix are also the characteristic multipliers of $\gamma$, and therefore determine the linear stability of $\gamma$.  In particular, $\gamma$ is \textit{spectrally stable} if all of its characteristic multipliers lie on the unit circle, and $\gamma$ is \textit{linearly stable} if it is spectrally stable and $X(T)$ is semisimple apart from trivial eigenvalues. \\

Additionally, if $Y$ is the fundamental matrix solution of
\begin{equation}\label{Linearized2}
\xi' = JD^2\Gamma(\gamma(s))\xi, \quad \xi(0) = Y_0
\end{equation}
for some invertible matrix $Y_0$, then by definition of $X(s)$, $Y(s) = X(s)Y_0$, implying $X(T) = Y(T)Y_0^{-1}$.  Then we have
$$X(T) = Y(T)Y_0^{-1} = Y_0(Y_0^{-1}Y(T))Y_0^{-1}$$
and so $X(T)$ and $Y_0^{-1}Y(T)$ are similar, and stability can be determined by the eigenvalues of either.\\

Linear stability is typically established by numerical integration.  Some elegant techniques for simplifying the numerical work by using the symmetries of an orbit were presented by Roberts in \cite{bibRoberts1}.  The relevant theory will be reviewed in Section \ref{Roberts}.

\subsection{Symmetries of the Orbit}\label{syms}

In order to utilize the techniques developed by Roberts, it is necessary to first identify the symmetries of the periodic orbit.

\begin{lemma}
The regularized Hamiltonian $\Gamma$ has symmetry group isomorphic to the Klein four group.
\end{lemma}

\begin{proof}
Let
\begin{equation*}
G = 
\begin{bmatrix}
1 & 0 \\ 0 & -1
\end{bmatrix},
\end{equation*}
and define the block matrix
\begin{equation}\label{sMatrix4df}
S = 
\begin{bmatrix}
-G & 0 & 0 & 0 \\
0 & -G & 0 & 0 \\
0 & 0 & G & 0 \\
0 & 0 & 0 & G
\end{bmatrix},
\end{equation}
where $0$ represents the $2 \times 2$ identity matrix.  Then we have
$$S^2 = (-S)^2 = I$$
Hence, $S$ and $-S$ generate a group isomorphic to the Klein four group.  For fixed values of $m$ and $E$, we have
$$\Gamma \circ (\pm S) = \Gamma$$
so $\pm S$ generate a Klein-four symmetry group for $\Gamma$ as well. \\
\end{proof}

Symmetries for $\Gamma$ also help to determine symmetries for the periodic orbit.

\begin{theorem}
Let $\gamma$ be a solution to the Hamiltonian system for some $m_1$, $m_2$ and $E < 0$ such that
$$\gamma(0) = (0, 0, 0, \zeta_4, 2m_1^{3/2}, 0, 0, 0)$$
and
$$\gamma(s_0) = (\zeta_1, 0, 0, 0, 0, 0, 0, \zeta_8)$$
for some values of $\zeta_1$, $\zeta_4$, and $\zeta_8$.  (In other words, $\gamma(0)$ corresponds to collision of the two bodies of mass $m_1$ somewhere along the $y$-axis away from the origin, and $\gamma(s_0)$ corresponds to a collinear arrangement of the three bodies along the $x$-axis.)  Then $\gamma$ extends to a $T = 4s_0$-periodic orbit, wherein $S$ and $-S$ are symmetries of the orbit.
\end{theorem}

\textit{Remark:} 
Although the $\Gamma = 0$ condition forces a relationship between $\zeta_1$ and $\zeta_8$, the most relevant detail is that both quantities are non-zero. \\

\begin{proof}
Note that if $\gamma(s)$ is a $T$-periodic solution to the regularized equations of motion resulting from the Hamiltonian $\Gamma$, a standard calculation shows that both $-S\gamma(T/2 - s)$ and $S\gamma(T - s)$ are solutions as well.  Existence and uniqueness of solutions then imply that
$$-S\gamma(T/2 - s) = \gamma(s) = S\gamma(T-s)$$
for all $s$.  Hence the symmetry group for Broucke's isosceles orbit is isomorphic to the Klein four group, with $S$ and $-S$ as generators.
\end{proof}

\subsection{Stability Reduction using Symmetries}\label{Roberts}

The following can be found in \cite{bibRoberts1}:

\begin{lemma}
Suppose $\gamma(s)$ is a $T$-periodic solution of a Hamiltonian system with Hamiltonian $\Gamma$ and a time-reversing symmetry $S$ such that:
\begin{enumerate}
\item[(i)] For some $n \in \mathbb{N}$, $\gamma(-s + T/N) = S(\gamma(s))$ for all $s$;
\item[(ii)] $\Gamma(Sz) = \Gamma(z)$;
\item[(iii)] $SJ = -JS$;
\item[(iv)] $S$ is orthogonal.
\end{enumerate}
Then the fundamental matrix solution $X(s)$ satisfies $$X(-s + T/N) = SX(s)S^T(X(T/N)).$$
\end{lemma}

\noindent Note that the matrix $S$ given in (\ref{sMatrix4df}) satisfies all the required hypotheses.

\begin{corollary}
Under the same hypotheses,
$$X(T/N) = SB^{-1}S^TB \text{ where } B = X(T/2N).$$
\end{corollary}

\begin{corollary}\label{corFactor}
If $Y(s)$ is the fundamental matrix solution to (\ref{Linearized2}), then
$$Y(-s + T/N) = SY(s)Y_0^{-1}S^TY(T/N)$$
and
$$Y(T/N) = SY_0B^{-1}S^TB \text{ where } B = Y(T/2N).$$
\end{corollary}

\noindent Similar results for time-preserving symmetries are also presented in \cite{bibRoberts1}, but are not needed for this orbit.  Using these results may allow the computation of the eigenvalues (hence stability) to be accomplished using only a fraction of the orbit.  Applying Corollary \ref{corFactor} with $N = 2$, $S$ as defined in (\ref{sMatrix4df}), and noting that $S^T = S$ yields
$$Y(T/2) = SY_0Y(T/4)^{-1}SY(T/4).$$
Similarly, if $N = 1$, since $S^2 = I$, we get
\begin{align*}
Y(T) &= -SY_0Y(T/2)^{-1}(-S)Y(T/2) \\
     &= SY_0[SY_0Y(T/4)^{-1}SY(T/4)]^{-1}S[SY_0Y(T/4)^{-1}SY(T/4)] \\   
     &= SY_0Y(T/4)^{-1}SY(T/4)Y_0^{-1}SY_0Y(T/4)^{-1}SY(T/4).
\end{align*}
This yields
\begin{align*}
Y_0^{-1}Y(T) &= Y_0^{-1}SY_0Y(T/4)^{-1}SY(T/4)Y_0^{-1}SY_0Y(T/4)^{-1}SY(T/4) \\
&= [Y_0^{-1}SY_0Y(T/4)^{-1}SY(T/4)]^2 \\
&= W^2
\end{align*}
with $W =Y_0^{-1}SY_0Y(T/4)^{-1}SY(T/4)$.  Hence, in order to analyze the stability of the orbit, we need only compute the entries of $Y$ along a quarter of the orbit. \\

Again, from \cite{bibRoberts1}:

\begin{lemma}
For a symplectic matrix $W$, suppose there is a matrix $K$ such that
\begin{equation*}
\frac{1}{2}(W + W^{-1}) =
\begin{bmatrix}
K^T & 0 \\
0 & K
\end{bmatrix}.
\end{equation*}
Then the eigenvalues of $W$ lie on the complex unit circle if and only if all of the eigenvalues of $K$ are real and lie in the interval $[-1,1]$
\end{lemma}

A matrix whose eigenvalues all lie on the complex unit circle is often referred to as \textit{stable}.

We now show that there is an appropriate choice of $Y_0$ for which $W$ has the required form, further reducing the stability calculations for the orbit.  If we let
\begin{equation*}
\Lambda =
\begin{bmatrix}
-I & 0 \\
0 & I 
\end{bmatrix},
\end{equation*}
then setting
\begin{equation}\label{y0form}
Y_0 =
\left[
\begin{array}{cccc|cccc}
%\begin{bmatrix}
0 & 0 & 0 & 0 & 1 & 0 & 0 & 0 \\
0 & 0 & 1 & 0 & 0 & 0 & 0 & 0 \\
0 & 0 & 0 & 0 & 0 & 1 & 0 & 0 \\
0 & 0 & 0 & 1 & 0 & 0 & 0 & 0 \\
\hline
-1 & 0 & 0 & 0 & 0 & 0 & 0 & 0 \\
0 & 0 & 0 & 0 & 0 & 0 & 1 & 0 \\
0 & -1 & 0 & 0 & 0 & 0 & 0 & 0 \\
0 & 0 & 0 & 0 & 0 & 0 & 0 & 1 
%\end{bmatrix}
\end{array}
\right]
\end{equation}
yields $-Y_0^{-1}SY_0 = \Lambda$.  (The lines here are provided for ease in reading.  Much of our later analysis will involve breaking $8 \times 8$ matrices down into $4 \times 4$ blocks.)  Furthermore, it is easy to check that $Y_0$ is both orthogonal and symplectic.  If we set $D = -B^{-1}SB$ for $B = Y(T/4)$, we then have
$$W = \Lambda D.$$
Also, since $\Lambda^2 = D^2 = I$, we know immediately that
$$W^{-1} = D\Lambda.$$
Since $B = Y(T/4)$ is symplectic, setting
\begin{equation*}
B =
\begin{bmatrix}
B_1 & B_2 \\
B_3 & B_4
\end{bmatrix}
\text{ and }
S =
\begin{bmatrix}
S_1 & 0 \\
0 & -S_1
\end{bmatrix}
\end{equation*}
gives
\begin{align*}
D &= -B^{-1}SB \\
&= -
\begin{bmatrix}
K^T & L_1 \\
-L_2 & K
\end{bmatrix},
\end{align*}
where $L_1$, $L_2$, and $K$ are $4 \times 4$ matrices satisfying $L_1 = B_4^TS_1B_2 + B_2^TS_1B_4$, $L_2 = B_3^TS_1B_1 - B_1^TS_1B_3$, and $K = -B_2^TS_1B_2 - B_1^TS_1B_4$.  Thus,
\begin{equation*}
W = \Lambda D = 
\begin{bmatrix}
K^T & L_1 \\
L_2 & K
\end{bmatrix}.
\end{equation*}
Similarly, we find that
\begin{equation*}
W^{-1} = D \Lambda = 
\begin{bmatrix}
K^T & -L_1 \\
-L_2 & K
\end{bmatrix}.
\end{equation*}
Thus, we have
\begin{equation}
\frac{1}{2}\left(W + W^{-1}\right) = 
\begin{bmatrix}
K^T & 0 \\
0 & K
\end{bmatrix}.
\label{WKrelate}
\end{equation}

\textit{Remark:}
The given matrix $Y_0$ in (\ref{y0form}) is not unique.  Different choices of $Y_0$ are possible, but our particular choice is helpful for much of our later analysis.  It is also worth noting that our choice of $Y_0$ is independent of the values of $m_1$ and $m_2$ for this orbit, which is not always true (see \cite{bibBMS}.)

We can give formulas for the entries of $K$ in terms of $W$.  Since $B$ is symplectic, we have $J = B^TJB$, and hence
$$B^{-1} = -JB^T J.$$
Using $W = \Lambda D$ for $D = -B^{-1}SB$ and the relation $-SJ = JS$, we find
\begin{align*}
W &= \Lambda (-B^{-1}SB) \\
&= \Lambda JB^T JSB \\
&= -\Lambda JB^T SJB.
\end{align*}
Directly computing $\Lambda J$ and using the block form of $B$, we find that
\begin{equation*}
(\Lambda J) B^T = -
\begin{bmatrix}
0 & I \\
I & 0
\end{bmatrix}
\begin{bmatrix}
B_1^T & B_3^T \\
B_2^T & B_4^T
\end{bmatrix}
= -
\begin{bmatrix}
B_2^T & B_4^T \\
B_1^T & B_3^T
\end{bmatrix}
.
\end{equation*}
Define $\text{col}_i(-SJB)$ to be the $i$th column of the matrix $-SJB$.  Then we have $\text{col}_i(-SJB) = -SGc_i$ where $c_i$ is the $i$th column of $B$.  Using the above two formulas, this implies that the $(i,j)$ entry of $W$ is given by $-c_i^T S J C_j$.  Equation (\ref{WKrelate}) shows that the $(i,j)$ entry of $K$ is the $(i+4, j+4)$ entry of $W$.  Hence, 
\begin{equation}
K = 
\begin{bmatrix}
-c_1^T SJc_5 & -c_1^T SJc_6 & -c_1^T SJc_7 & -c_1^T SJc_8\\
-c_2^T SJc_5 & -c_2^T SJc_6 & -c_2^T SJc_7 & -c_2^T SJc_8\\
-c_3^T SJc_5 & -c_3^T SJc_6 & -c_3^T SJc_7 & -c_3^T SJc_8\\
-c_4^T SJc_5 & -c_4^T SJc_6 & -c_4^T SJc_7 & -c_4^T SJc_8\\

\end{bmatrix}.
\label{Kentries}
\end{equation}

\textit{Remark:}
Computing the entries of $K$ this way will allow us to bypass computing $W^{-1}$.  This is preferred as a numerical method as $W$ may be very poorly conditioned.

\subsection{Entries of $K$ from Invariant Quantities}\label{Kapriori}

Before any numerical work is done, we can determine many of the values of entries of $K$ by using properties of the orbit.  We first introduce some notation to simplify the analysis.  Let $\mathcal{M}$ denote the set of matrices of the form
\begin{equation*}
\begin{bmatrix}
m_{11} & 0 & 0 & m_{14}\\
0 & m_{22} & m_{23} & 0\\
0 & m_{32} & m_{33} & 0\\
m_{41} & 0 & 0 & m_{44}\\
\end{bmatrix}
\end{equation*}
where all of the listed $m_{ij} \in \mathbb{R}$ ($m_{ij} = 0$ is permitted).  Further, let $\mathcal{M}_2$ denote the set of $8 \times 8$ matrices whose $4 \times 4$ blocks are in $\mathcal{M}$.  That is to say, $\mathcal{M}_2$ consists of matrices of the form
\begin{equation*}
\begin{bmatrix}
M_1 & M_2\\
M_3 & M_4\\
\end{bmatrix}
\end{equation*}
where each of the $M_i \in \mathcal{M}$. It is easy to check that $\mathcal{M}$ forms a ring with the standard definitions of matrix addition and multiplication.  Furthermore, each element of $\mathcal{M}$ with nonzero determinant has its inverse in $\mathcal{M}$ as well.  The same ring structure exists for $\mathcal{M}_2$.\\

The following two lemmas will help establish an important theorem about the form of $K$:

\begin{lemma}\label{closure1}
If $M \in \mathcal{M}_2$, then the system of differential equations given by
$$\eta' = M\eta$$
and initial conditions
$$\eta(0) = (*, 0, 0, *, *, 0, 0, *)^T$$
has solutions of the form
$$\eta(s) = (f_1(s), 0, 0, f_4(s), f_5(s), 0, 0, f_8(s))^T$$
\end{lemma}
\begin{proof}
We verify that $M\eta$ has the proper form.  Note that 
\begin{equation*}
\left[
\begin{array}{cccc|cccc}
* & 0 & 0 & * & * & 0 & 0 & * \\
0 & * & * & 0 & 0 & * & * & 0 \\
0 & * & * & 0 & 0 & * & * & 0 \\
* & 0 & 0 & * & * & 0 & 0 & * \\
\hline
* & 0 & 0 & * & * & 0 & 0 & * \\
0 & * & * & 0 & 0 & * & * & 0 \\
0 & * & * & 0 & 0 & * & * & 0 \\
* & 0 & 0 & * & * & 0 & 0 & *
\end{array}
\right]
\begin{bmatrix}
* \\
0 \\
0 \\
* \\
* \\
0 \\

0 \\
* \\
\end{bmatrix}
=
\begin{bmatrix}
* \\
0 \\
0 \\
* \\
* \\
0 \\
0 \\
* \\
\end{bmatrix}.
\end{equation*}
Hence, the zeros in the 2nd, 3rd, 6th, and 7th are preserved under multiplication by $M$.  So 
$$\eta(s) = (f_1(s), 0, 0, f_4(s), f_5(s), 0, 0, f_8(s))^T$$
is a solution of $\eta' = M\eta$.  Existence and uniqueness of solutions implies that $\eta(s)$ is the only solution of the system.
\end{proof}

\begin{lemma}\label{closure2}
If $M \in \mathcal{M}_2$, then the system of differential equations given by
$$\eta' = M\eta$$
and initial conditions
$$\eta(0) = (0, *, *, 0, 0, *, *, 0)^T$$
has solutions of the form
$$\eta(s) = (0, f_2(s), f_3(s), 0, 0, f_6(s), f_7(s), 0)^T.$$
\end{lemma}

\begin{theorem}\label{kFormTheorem}
If $Y_0$ is given by (\ref{y0form}), then $K \in \mathcal{M}$.
\end{theorem}

\begin{proof}
Using a computer algebra system, we find that the matrix $JD^2\Gamma$ is of the form
\begin{equation*}
%\begin{bmatrix}
\begin{bmatrix}
O & I \\
-I & O \\
\end{bmatrix}
\left[
\begin{array}{cccc|cccc}
* & \omega & \omega & * & 0 & 0 & \omega & * \\
\omega & * & * & \omega & 0 & 0 & \omega & \omega \\
\omega & * & * & \omega & 0 & 0 & 0 & 0 \\
* & \omega & \omega & * & 0 & 0 & 0 & 0 \\
\hline
0 & 0 & 0 & 0 & * & 0 & 0 & 0 \\
0 & 0 & 0 & 0 & 0 & * & 0 & 0 \\
\omega & \omega & 0 & 0 & 0 & 0 & * & 0 \\
* & \omega & 0 & 0 & 0 & 0 & 0 & *
%\end{bmatrix}
\end{array}
\right]
=
\left[
\begin{array}{cccc|cccc}
0 & 0 & 0 & 0 & * & 0 & 0 & 0 \\
0 & 0 & 0 & 0 & 0 & * & 0 & 0 \\
\omega & \omega & 0 & 0 & 0 & 0 & * & 0 \\
* & \omega & 0 & 0 & 0 & 0 & 0 & * \\
\hline
* & \omega & \omega & * & 0 & 0 & \omega & * \\
\omega & * & * & \omega & 0 & 0 & \omega & \omega \\
\omega & * & * & \omega & 0 & 0 & 0 & 0 \\
* & \omega & \omega & * & 0 & 0 & 0 & 0
\end{array}
\right].
\end{equation*}
Here, the zeros denote entries for which the mixed partials evaluate to zero identically, and the entries denoted $\omega$ are entries for which the mixed partials evaluate to zero assuming the conditions in $\mathcal{A}$ which hold along the periodic orbit $\gamma(s)$.  Under such conditions, we have $JD^2\Gamma \in \mathcal{M}_2$.  If $\xi(0) \in \mathcal{M}_2$, then each of the columns of $\xi(0)$ has the same form as in either Lemma \ref{closure1} or \ref{closure2}.  Hence, the solution to the system of linearized equations given by (\ref{Linearized2}) satisfies $\xi(s) \in \mathcal{M}_2$ for all $s$.  Since $Y_0 \in \mathcal{M}_2$, $\xi(s) \in \mathcal{M}_2$ for all $s$.  Lastly, because $M \in \mathcal{M}_2$ implies $M^{-1} \in \mathcal{M}_2$ when $M^{-1}$ exists, Equation (\ref{WKrelate}), gives $K \in \mathcal{M}$ as claimed.
\end{proof}

\textit{Remarks:}
\begin{enumerate}
\item[(i)] The structure of $\mathcal{M}_2$ very nicely decomposes phase space into a direct sum of $\mathcal{A} = \{Q_2 = Q_3 = P_2 = P_3 = 0\}$ and $\mathcal{A}^\perp = \{Q_1 = Q_4 = P_1 = P_4 = 0\}$.  This decomposition is due in part to the coordinate transformation we chose.  The choice of notation for $\mathcal{A}^\perp$ is appropriate in that $\mathcal{A}^\perp$ and $\mathcal{A}$ are orthogonal complements in $\mathbb{R}^8$.  The two are also skew-orthogonal: if $a_1 \in \mathcal{A}$ and $a_2 \in \mathcal{A}^\perp$, then $a_1^TJa_2 = 0$.

\item[(ii)] Matrices of the form $\mathcal{M}$ and $\mathcal{M}_2$ are similar to the diamond product discussed in \cite{bibLong}.  Specifically, $\Sigma^{-1} M \Sigma = \mathcal{A}_1 \Diamond \mathcal{A}_2$ for some matrices $\mathcal{A}_1$ and $\mathcal{A}_2$, where $M \in \mathcal{M}_2$ and $\Sigma$ is a (not necessarily unique) permutation matrix.  Furthermore, one of the $\mathcal{A}_i$ (depending on choice of $\Sigma$) corresponds to the 2DF setting.

\item[(iii)] The particular choice of $Y_0$ given in (\ref{y0form}) is important for this argument.
\end{enumerate}

In light of Theorem \ref{kFormTheorem}, we need only to find eight of the entries of the matrix $K$.  We can, in fact, reduce this number further by using invariant properties of the orbit $\gamma(s)$.  As is well-known, invariant quantities of the $n$-body problem are center of mass, net momentum, angular momentum, and energy (the Hamiltonian itself).  Each of these correspond to trivial eigenvalues of the monodromy matrix.  The center of mass and net momentum were ``factored out'' by our choice of coordinates at the beginning.  The remaining two invariant quantities will be used to reduce the number of entries of $K$ needed to find its eigenvalues.  

\begin{theorem}
The matrix $K$ has a right eigenvector $[1,0,0,0]^T$, corresponding to eigenvalue $-1$.
\label{ev1}
\end{theorem}

\begin{proof}
Let $v = Y_0^{-1}\gamma^{\ '}(0)/||\gamma^{\ '}(0)||$ or, equivalently, $Y_0^T\gamma^{\ '}(0)/||\gamma^{\ '}(0)||$.  By Corollary \ref{corFactor}, since $Y_0$ is orthogonal and $S$ is symmetric, we have
$$W = Y_0^{-1}SY_0B^{-1}SB = Y_0^{-1}SY_0B^{-1}S^TB = Y_0^TY(T/2).$$
Since $\gamma^{\ '}(s)$ is a solution of $\dot{\xi} = JD^2\Gamma(\gamma(s))\xi$ and $\gamma^{\ '}(0) = Y(0)Y_0^{-1}\gamma^{\ '}(0) = Y(0)v$, we also know that $\gamma^{\ '}(s) = Y(s)Y_0^{-1}\gamma^{\ '}(0) = Y(s)v$.  This implies
\begin{equation}\label{eigeq1}
Y_0^{-1}\gamma^{\ '}(T/2) = Y_0^TY(T/2)v = Wv.
\end{equation}
By the symmetry $\gamma(s) = -S\gamma(T/2 - s)$, we also have $\gamma^{\ '}(s) = S\gamma^{\ '}(T/2 - s)$.  Setting $s = 0$ in this setting tells us that $\gamma^{\ '}(0) = S\gamma^{\ '}(T/2)$.  Since
$$\gamma^{\ '}(0) = (\omega, 0, 0, 0, 0, 0, 0, 0)$$
for some real number $\omega$, we have $-S\gamma^{\ '}(0) = \gamma^{\ '}(0)$.  Thus
\begin{equation}\label{eigeq2}
Y_0^{-1}\gamma^{\ '}(T/2) = Y_0^TS\gamma^{\ '}(0) = -Y_0^T\gamma^{\ '}(0) = -v.
\end{equation}
Combining (\ref{eigeq1}) and (\ref{eigeq2}) gives $Wv = -v$, and so $-1$ is an eigenvalue of $W$ with eigenvector $v$.  By definition, we have that
$$v = Y_0^T\gamma^{\ '}(0)/||\gamma^{\ '}(0)|| = (0, 0, 0, 0, 1, 0, 0, 0).$$
From the form of $W$, this implies that
\begin{equation*}
K
\begin{bmatrix}
1 \\
0 \\
0 \\
0
\end{bmatrix}
=
\begin{bmatrix}
-1 \\
0 \\
0 \\
0 \\
\end{bmatrix}.
\end{equation*}
So $[1, 0, 0, 0]^T$ is an eigenvector of $K$ with eigenvalue $-1$.  Consequently, the first column of $K$ must be $[-1,0,0,0]^T$. \\
\end{proof}

Combining the results of the previous two theorems gives
\begin{equation}\label{Kform}
K = 
\begin{bmatrix}
-1 & 0 & 0 & * \\
0 & a & b & 0 \\
0 & c & d & 0 \\
0 & 0 & 0 & e
\end{bmatrix}.
\end{equation}

\textit{Remark:} Owing to the decomposition of linearized phase space into two invariant subspaces and the ordering of the coordinates, the position of $e$ in the matrix $K$ indicates that it should be an eigenvalue corresponding to the behavior of the orbit in $\mathcal{A}$.  This eigenvalue, along with the trivial eigenvalue $-1$ from the $(1,1)$ position, completely classify linear stability in the 2DF setting.  Hence, computing the linear stability of the 4DF orbit in the chosen coordinates automatically gives the stability of the 2DF orbit.  (The results will be discussed further in Section \ref{stabilities}.) \\

We can make use of the final invariant quantity, angular momentum, to further simplify our calculations.  This is an extension of Roberts' method from \cite{bibRoberts1}, in which coordinate transformations ``factor out'' the angular momentum before linearization is performed.  In the following theorem, we are able to show that this invariant quantity can be used to simplify linear stability calculations after linearization.

\begin{theorem}
The matrix $K$ has a left eigenvector $\nabla A(\gamma(0))Y_0$ with eigenvalue $-1$.
\label{ev2}
\end{theorem}

\begin{proof}
This proof is based on ideas given in \cite{bibMHO}, p. 134, Lemma 7.  Define $\hat{v}(s) = \nabla A(\gamma(s))$, where $A$ represents the regularized angular momentum given in (\ref{angMomFinal}).  Then 
$$\hat{v} = \left(\frac{1}{2}P_2, -\frac{1}{2}P_1, 2\mu P_4, -2\mu P_3, -\frac{1}{2}Q_2, \frac{1}{2}Q_1, -2\mu Q_4, 2\mu Q_3 \right).$$
Since
$$\gamma(0) = (0,0,0,\zeta_4,2m_1^{3/2},0,0,0),$$
we know that
$$\hat{v}(0) = \left(0,-m_1^{3/2}, 0, 0, 0, 0, -2\mu\zeta_4, 0\right).$$

Let $\phi(s,z)$ be the general solution to the system of regularized differential equations with initial condition $z$.  Then
$$A(\phi(s,z)) = A(z).$$
Differentiating with respect to $z$ gives
$$\nabla A(\phi(s,z))\frac{\partial \phi}{\partial z}(s,z) = \nabla A(z)$$
or, equivalently
$$\hat{v}(s)X(s) = \hat{v}(0)$$
where $X(s)$ is the fundamental matrix solution.  Setting $s = T/2$ and substituting $X(T/2) = Y_0(Y_0^{-1}Y(T/2))Y_0^{-1}$ gives
$$\hat{v}(T/2)Y_0(Y_0^{-1}Y(T/2))Y_0^{-1} = \hat{v}(0)$$
and so
$$\hat{v}(T/2)Y_0(Y_0^{-1}Y(T/2)) = (\hat{v}(T/2)Y_0)W= \hat{v}(0)Y_0.$$
By the symmetry of the orbit, $\gamma(T/2) = -\gamma(0)$, which gives
$$\hat{v}(T/2) = -\hat{v}(0),$$
and therefore
$$(\hat{v}(0)Y_0)W= -\hat{v}(0)Y_0.$$
Hence $\hat{v}(0)Y_0$ is a left eigenvector for $W$ with eigenvalue $-1$.
\end{proof}

\textit{Remarks:}
\begin{enumerate}
\item[(i)] Since $K$ is real-valued, this result, along with other results about the form of $K$, force all of the eigenvalues of $K$ to be real.
\item[(ii)] This analysis is an improvement over work done in \cite{bibBMS}, in which the $-1$ eigenvalue corresponding to angular momentum showed up numerically but could not be factored out \textit{a priori}.
\end{enumerate}

To find the eigenvector, we readily compute $\hat{v}(0)Y_0 = (0,-2\mu\zeta_4,m_1^{3/2},0,0,0,0,0)$.  From this, we know that
$$(0,-2\mu\zeta_4,m_1^{3/2},0)K^T = - (0,-2\mu\zeta_4,m_1^{3/2},0).$$
Since $K \in \mathcal{M}$, this requires that the additional $-1$ eigenvalue comes from the central $2 \times 2$ block in $K$.  Furthermore, this imposes some relations on the entries $a, b, c, d$ in (\ref{Kform}).  In particular,
$$-2\mu\zeta_4 a + m_1^{3/2} b = 2\mu \zeta_4$$
$$-2\mu\zeta_4 c + m_1^{3/2} d = -m_1^{3/2} $$
Since Theorems \ref{ev1} and \ref{ev2} account for two of the eigenvalues of $K$, we have shown that we only need to find three entries of the matrix $K$ in order to determine the linear stability of the orbit.  Specifically, finding $e$ gives the linear stability in the 2DF setting.  Combining this result with one element of $\{a,b\}$ and one element of $\{c,d\}$ allows us to determine the second eigenvalue corresponding to linear stability in the 4DF setting. \\

\section{Results}\label{Numerics}

\subsection{Initial Conditions of the Orbit}

Recall that the initial conditions for the orbit in the regularized setting determined by $\Gamma$ are
$$\gamma(0) = (0,0,0,\zeta_4,2m_1^{3/2},0,0,0).$$
For $m_1 = 0$, the corresponding orbit consists of a body of mass 3 at the origin and two massless particles on the $x$-axis.  If $m_1 = 1.5$, the corresponding orbit contains a massless particle at the origin and two bodies of mass 1.5 on the $x$-axis.  Note that any vertical motion of the positive-mass particle(s) in either case would cause the center of mass to drift away from the origin, so the orbit must remain collinear for all time.  Hence, we expect that at as $m_1 \to 0^+$ or $m_1 \to 1.5^-$, the value of $Q_4(0)$ should approach 0. \\

Setting $E = -1$, and using numerical methods, we find the value of $Q_4(0)$.  The results of that calculation are shown in Figure \ref{Q4plot}.  Numerical difficulties occurred as $m_1 \to 1.5^-$ (equivalently $m_2 \to 0^+$), likely due to the prevalence of $m_1/m_2$ terms occurring in the equations of motion. (In fact, numerical methods were not able to find initial conditions for values of $m_1 \geq 1.47$.)\\
\begin{figure}[h]
\begin{center}
\includegraphics[scale=0.6]{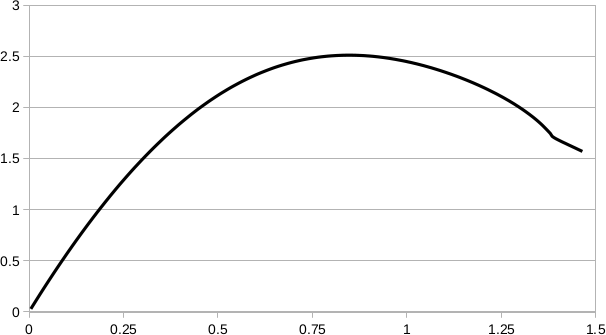}
\caption{A plot of the value of $Q_4(0)$.  The value of $m_1$ is plotted on the horizontal axis.}
\label{Q4plot} 
\end{center}
\end{figure}

\subsection{Stability in Two Settings}\label{stabilities}

As discussed in Section \ref{aExist}, in order to find the initial conditions for the 4DF orbit, we need only to find the initial conditions for the 2DF orbit.  This is done numerically for values of $m_1 = 0.005, 0.010, 0.015, ..., 1.485, 1.490, 1.495$.  

We numerically obtain the matrix $W$ (hence $K$) by a numerical integration of the linearized systems using the initial conditions discussed previously.  The values of $a$, $d$, and $e$ in the matrix $K$, as given in (\ref{Kform}), are readily computed using (\ref{Kentries}).  Knowing these, we are able to determine the eigenvalues of $K$.  The results are represented in Figures \ref{k11plot} through \ref{eandtr}. \\

\begin{figure}[h]
\begin{center}
\includegraphics[scale=.6]{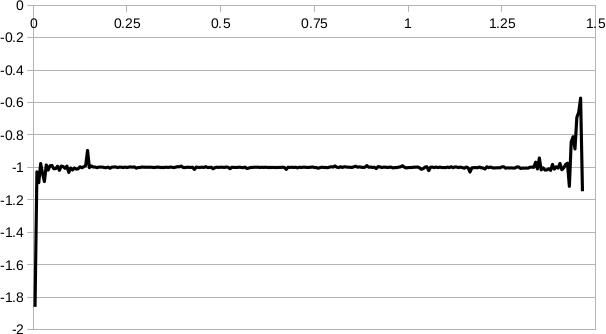}
\caption{A plot of the numerically-computed upper-left entry $k_{11}$ from the matrix $K$ in Equation \ref{Kform} (vertical) against $m_1$ (horizontal).}
\label{k11plot} 
\end{center}
\end{figure}

\begin{figure}[h]
\begin{center}
\includegraphics[scale=.6]{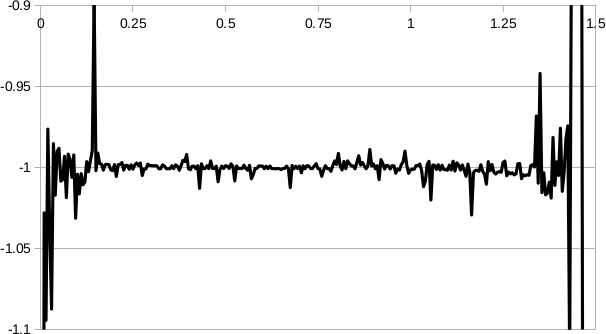}
\caption{A vertically rescaled plot of the graph from Figure \ref{k11plot}.}
\label{k11plotzoom} 
\end{center}
\end{figure}

Figures \ref{k11plot} and \ref{k11plotzoom} are used as verification of the numerical method, comparing the proven value $k_{11} = -1$ against the numerically computed value. \\

\begin{figure}[h]
\begin{center}
\includegraphics[scale=.6]{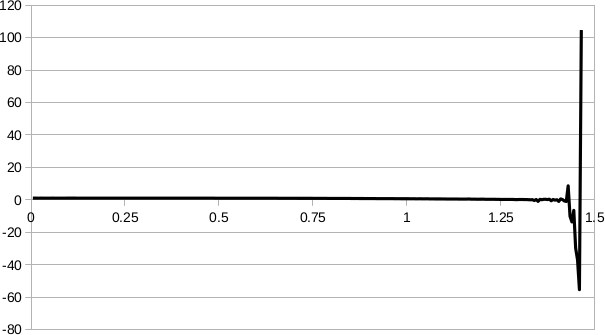}
\caption{A plot of the value of $e$ from Equation (\ref{Kform}) corresponding to stability in the 2DF setting.}
\label{stabplot2full} 
\end{center}
\end{figure}

\begin{figure}[h]
\begin{center}
\includegraphics[scale=.6]{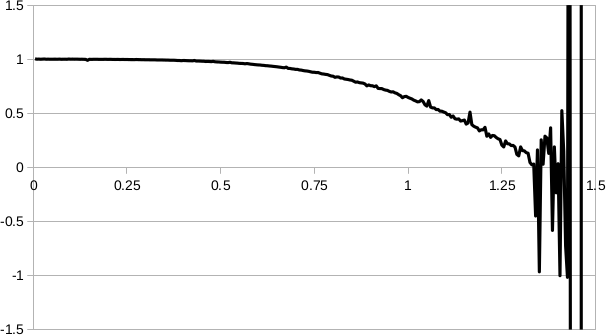}
\caption{A vertically rescaled plot of the value of $e$.  The curve appears to be asymptotic to the line $y = 1$ as $m_1 \to 0^-$.}
\label{stabplot2zoom} 
\end{center}
\end{figure}

In the 2DF setting, the numerical result is that $|e| < 1$ for all values of $m_1$ computed between $0.125$ through $1.4$ inclusive.  However, for values of $m_1 > 1.295$, the curve $y = e(m_1)$ is no longer monotonically decreasing, which corresponds to the sharp bend in Figure \ref{Q4plot}.  (This sharp bend and the particular numerical difficulties will be further discussed in an upcoming paper.) \\

Following these calculations, we obtain the following:
\begin{theorem}
Broucke's orbit in the 2DF setting is linearly stable for $m_1$ lying in an open interval containing $[0.125, 1.295]$.
\end{theorem}

For $m_1 \in [0.01, 1.25]$, the computed values of $e$ all lie within the interval $[0.998, 1.002]$, with the computed values of $k_{11}$ in the interval $[0.905, 1.095]$.  Allowing for minor numerical error, we propose the following:

\begin{conjecture}
Broucke's orbit in the 2DF setting is linearly stable for all $m_1$ in an open interval containing $[0.01, 0.125]$.
\end{conjecture}

To compute stability in the 4DF setting, we find the second eigenvalue of the central $2 \times 2$ block of $K$.  From Theorem \ref{ev2}, we know that one eigenvalue of this block is $-1$.  Hence, using the property that the trace of a matrix is equal to the sum of the eigenvalues, we know the second eigenvalue is given by $a + d + 1$, with $a$ and $d$ as in Equation (\ref{Kform}).  This value is plotted in Figures \ref{trfull} and \ref{trfullzoom}. \\

\textit{Remark:}
Computing the second eigenvalue could also be done calculating the negative of the determinant of the central $2 \times 2$ block.  Both were done during the research.  Both yielded the same stability results. \\

\begin{figure}[h]
\begin{center}
\includegraphics[scale=.6]{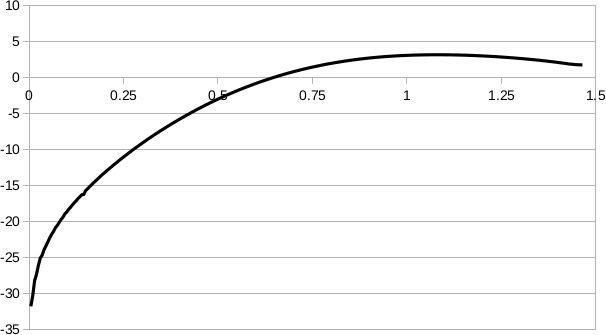}
\caption{A plot of the second non-trivial eigenvalue of $K$.}
\label{trfull} 
\end{center}
\end{figure}

\begin{figure}[h]
\begin{center}
\includegraphics[scale=.6]{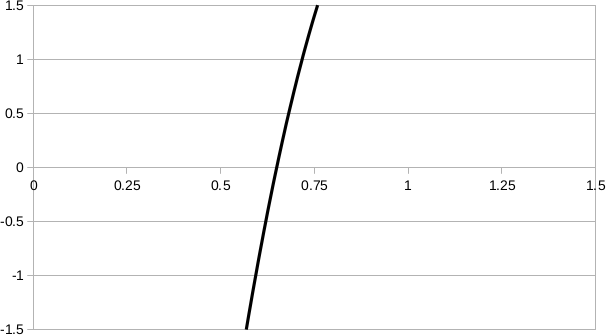}
\caption{A vertically rescaled plot of the value of the second non-trivial eigenvalue of $K$.  For the values of $m_1 \in [0.595, 0.715]$, this second eigenvalue lies within the interval $[-1,1]$.}
\label{trfullzoom} 
\end{center}
\end{figure}

\begin{figure}[h]
\begin{center}
\includegraphics[scale=.6]{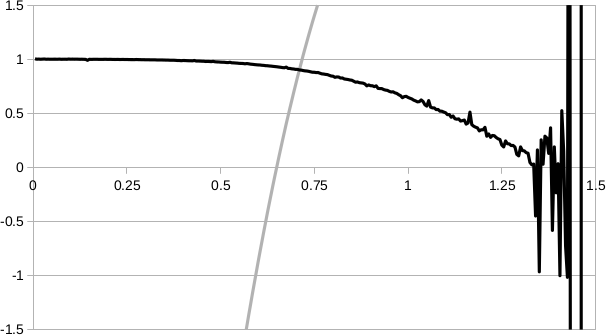}
\caption{Both eigenvalue plots superimposed.  The two graphs cross at roughly $m_1 = 0.71$, where the difference between the two is numerically $0.000429$.}
\label{eandtr} 
\end{center}
\end{figure}

These results give the following:

\begin{theorem}
Broucke's orbit in the 4DF setting is spectrally stable for $m_1$ lying in an open interval containing $[0.595, 0.715]$.
\label{4dfSpec}
\end{theorem}

We also note that there are possible values of $m_1$ for which we establish \textbf{only} spectral stability, due to repeated eigenvalues on the unit circle.  For simplicity of explanation, let $f(m)$ denote the lighter curve in Figure \ref{eandtr}, and $g(m)$ denote the darker curve. Also let $I$ be the interval specified by Theorem \ref{4dfSpec}.  Roberts' argument (see \cite{bibRoberts1}) demonstrates that each of the computed eigenvalues of $K$ in $[-1,1]$ correspond to the real part of a square root of an eigenvalue on the complex unit circle.  Accordingly, the value of $m_1 = \mathbf{m_1}$ where $f(\mathbf{m_1}) = g(\mathbf{m_1})$ is a point with duplicated eigenvalues, hence only spectral stability.  As noted above, $\mathbf{m_1} \approx 0.71$.  Similarly, a value of $\mathbf{m_2}$ satisfying $g(\mathbf{m_2}) = \pm 1$ would give a pair of eigenvalues of $1$ in $W^2$.  (If this curve is asymptotic to the line $y = 1$ as conjectured, this may not exist.  Certainly no such point exists for $m_1 \in I$)  For $\mathbf{m_3}$ with $g(\mathbf{m_3}) = 0$, we get $(\pm i)^2 = -1$ eigenvalues of $W^2$.  However, this also does not occur for $m_1 \in I$.  Finally, there is a fourth possibility when $\cos(2f(\mathbf{m_4})) = \cos(2g(\mathbf{m_4}))$, which arises by equating the real parts of $(e^{i\pi \theta_1})^2 = (e^{i\pi \theta_2})^2$ when $\theta_1 \neq \theta_2$.  Within $m_1 \in I$, this occurs for some value between 0.595 and 0.6, and again at the value of $\mathbf{m_1}$ where the two curves meet.  Owing to these results, we conclude with the following: \\  

\begin{theorem}
Apart from the two values of $m_1$ listed above, Broucke's orbit is linearly stable in the same interval given by Theorem \ref{4dfSpec}.
\end{theorem}

%We recall again that if we restrict to $\mathcal{A}$, then the eigenvalue of $K$ given by the $(4,4)$ entry corresponds to linear stability of the 2DF orbit.  This value stays in the interval $[-1,1]$ for $m_1 = .110, .115, ..., 1.490, 1.495$.  Hence,
%\begin{corollary}\label{2DFStab}
%Broucke's orbit in the 2DF setting is linearly stable for all $m_1 \in (\epsilon_1, 1.5-\epsilon_2)$ for some $0 < \epsilon_1 \leq 0.11$ and $\epsilon_2 > 0$.
%\end{corollary}
%It is worth noting that numerically this eigenvalue lies in the interval $-1 \pm .0015$ for values of $m_1 < .11$.  Hence it may be the case that the orbit is linearly stable in the 2DF setting lower values of $m_1$.
%\begin{conjecture}
%Broucke's orbit in the 2DF setting is linearly stable for all $m_1 \in (0,1.5)$.
%\end{conjecture}

\bibliographystyle{plain}
\bibliography{BrouckeNotes}

\begin{thebibliography}{10}

\bibitem{bibBSrhomb}
Lennard Bakker and Skyler Simmons.
\newblock Stability of the rhomboidal symmetric-mass orbit.
\newblock {\em Discrete Contin. Dyn. Syst.}, 35(1):1--23, 2015.

\bibitem{bibBMS}
Lennard~F. Bakker, Scott Mancuso, and Skyler~C. Simmons.
\newblock Linear stability for some symmetric periodic simultaneous binary
  collision orbits in the planar pairwise symmetric four-body problem.
\newblock {\em J. Math. Anal. Appl.}, 392(2):136--147, 2012.

\bibitem{bibBOYS1}
Lennard~F. Bakker, Tiancheng Ouyang, Duokui Yan, and Skyler Simmons.
\newblock Existence and stability of symmetric periodic simultaneous binary
  collision orbits in the planar pairwise symmetric four-body problem.
\newblock {\em Celestial Mech. Dynam. Astronom.}, 110(3):271--290, 2011.

\bibitem{bibBOYS2}
Lennard~F. Bakker, Tiancheng Ouyang, Duokui Yan, and Skyler Simmons.
\newblock Erratum to: {E}xistence and stability of symmetric periodic
  simultaneous binary collision orbits in the planar pairwise symmetric
  four-body problem [mr2821623].
\newblock {\em Celestial Mech. Dynam. Astronom.}, 112(4):459--460, 2012.

\bibitem{bibBORSY}
Lennard~F. Bakker, Tiancheng Ouyang, Duokui Yan, Skyler Simmons, and Gareth~E.
  Roberts.
\newblock Linear stability for some symmetric periodic simultaneous binary
  collision orbits in the four-body problem.
\newblock {\em Celestial Mech. Dynam. Astronom.}, 108(2):147--164, 2010.

\bibitem{bibBrouckeOriginal}
Roger Broucke.
\newblock On the isosceles triangle configuration in the planar general three
  body problem.
\newblock {\em Astron. Astrophys.}, 73(3):303--313, 1979.

\bibitem{bibHenon}
M.~H\'{e}non.
\newblock Stability of interplay oribts.
\newblock {\em Cel. Mech.}, 15:243--261, 1977.

\bibitem{bibHM1}
Jarmo Hietarinta and Seppo Mikkola.
\newblock Chaos in the one-dimensional gravitational three-body problem.
\newblock {\em Chaos}, 3(2):183--203, 1993.

\bibitem{bibLong}
Yiming Long.
\newblock {\em Index theory for symplectic paths with applications}, volume 207
  of {\em Progress in Mathematics}.
\newblock Birkh\"auser Verlag, Basel, 2002.

\bibitem{bibMartinez}
Regina Mart{\'{\i}}nez.
\newblock On the existence of doubly symmetric ``{S}chubart-like'' periodic
  orbits.
\newblock {\em Discrete Contin. Dyn. Syst. Ser. B}, 17(3):943--975, 2012.

\bibitem{bibMcGehee}
Richard McGehee.
\newblock Triple collision in the collinear three-body problem.
\newblock {\em Invent. Math.}, 27:191--227, 1974.

\bibitem{bibMHO}
Kenneth~R. Meyer, Glen~R. Hall, and Dan Offin.
\newblock {\em Introduction to {H}amiltonian dynamical systems and the
  {$N$}-body problem}, volume~90 of {\em Applied Mathematical Sciences}.
\newblock Springer, New York, second edition, 2009.

\bibitem{bibMoeckel}
Richard Moeckel.
\newblock A topological existence proof for the {S}chubart orbits in the
  collinear three-body problem.
\newblock {\em Discrete Contin. Dyn. Syst. Ser. B}, 10(2-3):609--620, 2008.

\bibitem{bibMoeckelMontgomery}
Richard Moeckel and Richard Montgomery.
\newblock Symmetric regularization, reduction and blow-up of the planar
  three-body problem.
\newblock {\em Pacific J. Math.}, 262(1):129--189, 2013.

\bibitem{bibYan3}
Tiancheng Ouyang and Duokui Yan.
\newblock Periodic solutions with alternating singularities in the collinear
  four-body problem.
\newblock {\em Celestial Mech. Dynam. Astronom.}, 109(3):229--239, 2011.

\bibitem{bibOYS}
Tiancheng Ouyang, Duokui Yan, and Skyler Simmons.
\newblock Periodic solutions with singularities in two dimensions in the
  $n$-body problem.
\newblock {\em Rocky Mtn. J. Math.}, 42(4):1601--1614, 2012.

\bibitem{bibRoberts1}
Gareth~E. Roberts.
\newblock Linear stability analysis of the figure-eight orbit in the three-body
  problem.
\newblock {\em Ergodic Theory Dynam. Systems}, 27(6):1947--1963, 2007.

\bibitem{bibSchubart}
J.~Schubart.
\newblock Numerische {A}ufsuchung periodischer {L}\"osungen im
  {D}reik\"orperproblem.
\newblock {\em Astr. Nachr.}, 283:17--22, 1956.

\bibitem{bibShib1}
Mitsuru Shibayama.
\newblock Minimizing periodic orbits with regularizable collisions in the
  {$n$}-body problem.
\newblock {\em Arch. Ration. Mech. Anal.}, 199(3):821--841, 2011.

\bibitem{bibSweatman1}
Winston~L. Sweatman.
\newblock The symmetrical one-dimensional {N}ewtonian four-body problem: a
  numerical investigation.
\newblock {\em Celestial Mech. Dynam. Astronom.}, 82(2):179--201, 2002.

\bibitem{bibSweatman2}
Winston~L. Sweatman.
\newblock A family of symmetrical {S}chubart-like interplay orbits and their
  stability in the one-dimensional four-body problem.
\newblock {\em Celestial Mech. Dynam. Astronom.}, 94(1):37--65, 2006.

\bibitem{bibVenturelli}
Andrea Venturelli.
\newblock A variational proof of the existence of von {S}chubart's orbit.
\newblock {\em Discrete Contin. Dyn. Syst. Ser. B}, 10(2-3):699--717, 2008.

\bibitem{bibYan1}
Duokui Yan.
\newblock Existence and linear stability of the rhomboidal periodic orbit in
  the planar equal mass four-body problem.
\newblock {\em J. Math. Anal. Appl.}, 388(2):942--951, 2012.

\bibitem{bibBrockeYan}
Duokui Yan.
\newblock Existence of the {B}roucke periodic orbit and its linear stability.
\newblock {\em J. Math. Anal. Appl.}, 389(1):656--664, 2012.

\end{thebibliography}

\end{document}